\documentclass{amsart}
\usepackage{graphicx,amsmath,amsthm,amssymb, amsfonts}
\usepackage[all]{xy}
\newtheorem{definition}{Definitions}[section]

\newtheorem{lemma}[definition]{Lemma}
\newtheorem{prop}[definition]{Proposition}
\newtheorem{theo}[definition]{Theorem}
\newtheorem{Coro}[definition]{Corollary}
\newtheorem{remark}[definition]{Remark}

\newcommand{\Ext}{{\rm Ext}}
\newcommand{\Tor}{{\rm Tor}}

\newcommand{\Hom}{{\rm Hom}}
\newcommand{\Spec}{{\rm Spec}}

\newcommand{\Supp}{{\rm Supp}}

\newcommand{\Max}{{\rm Max}}

\newcommand{\fp}{{\frak p}}

\newcommand{\fm}{{\frak m}}

\newcommand{\fa}{{\frak a}}
\newcommand{\fb}{{\frak b}}

\newcommand{\dlim}{{\displaystyle\lim_{\stackrel{\longrightarrow}{\scriptscriptstyle{n\in\Bbb{N}}}}}}
\begin{document}

\title[Cominimaxness with respect to ideals of dimension one]
{Cominimaxness with respect to ideals of dimension one}%

\author[Roshan]{Hajar Roshan-Shekalgourabi}%
\address{Department of Basic Sciences, Arak University of Technology, P. O. Box 38135-1177, Arak,  Iran.}%
\email{hrsmath@gmail.com and Roshan@arakut.ac.ir}%


\date{\today}%
\begin{abstract}
Let $R$ be a commutative Noetherian ring, $\fa$ be an ideal of $R$ and $M$ be an $R$-module. It is shown that if $\Ext^i_R(R/\fa,M)$ is minimax for all $i\leq \dim M$, then the $R$-module $\Ext^i_R(N,M)$ is minimax for all $i\geq 0$ and for any finitely generated $R$-module $N$ with $\Supp_R(N) \subseteq V (\fa)$ and $\dim N \leq 1$. As a consequence of this result we obtain that for any $\fa$-torsion $R$-module $M$ that $\Ext^i_R(R/\fa, M)$ is minimax for all $i\leq \dim M$, all Bass numbers and all Betti numbers of $M$ are finite. This generalizes \cite[Corollary 2.7]{BNS2015}. Also, some equivalent conditions for the cominimaxness of local cohomology modules with respect to ideals of dimension at most one are given.
\end{abstract}
\maketitle

\section{Introduction}
Let $R$ denote a commutative Noetherian ring with identity and $\fa$ be an ideal of $R$. For an $R$-module $M$, the $i$th local cohomology module of $M$ with respect to $\fa$ is defined as $$H^i_\fa(M)\cong\dlim \Ext^i_R(R/\fa^n, M).$$ For more details about the local cohomology, we refer the reader to \cite{BS}.

In 1968, Grothendieck \cite{Gro69} conjectured that for any ideal $\fa$ of $R$ and any finitely generated $R$-module $M$, $\Hom_R \left(R/\fa, H^i_\fa (M)\right)$ is a finitely generated $R$-module for all $i$. One year later, by proving a counterexample, Hartshorne \cite{Ha} showed that the Grothendieck's conjecture is not true in general even $R$ is regular and introduced the class of cofinite modules with respect to an ideal. He defined an $R$-module $M$ to be \emph{$\fa$-cofinite} if $\Supp_R(M) \subseteq V(\fa)$ and $\Ext^j_R (R/\fa,M)$ is finitely generated for all $j$ and posed the following question:

\begin{itemize}
  \item For which rings $R$ and ideals $\fa$ is the module $H^i_\fa (M)$ $\fa$-cofinite for all $i$ and all finitely generated $R$-modules $M$?\label{Q1}
\end{itemize}

There are many papers that are devoted to study this question. For example, see \cite{Ha, Chir, DM, Y, BN}.

In \cite{Zo}, Z\"{o}schinger introduced the interesting class of minimax modules, modules containing some finitely generated submodule such that the  quotient module is Artinian. As a generalization of the concept of $\fa$-cofinite modules, the concept of $\fa$-cominimax modules was introduced in \cite{ANV}. An $R$-module $M$ is said to be  \emph{$\fa$-cominimax} if $\Supp_R(M) \subseteq V(\fa)$ and $\Ext^j_R (R/\fa,M)$ is minimax for all $j$. Since the concept of minimax modules is a natural generalization of the concept of finitely generated modules, many authors studied the minimaxness of local cohomology modules and answered the Hartshorn's question in the class of minimax modules (see \cite{AR, ARH, ANV, M2011}).

As an important problem in commutative algebra is determining when the Bass numbers and Betti numbers of a module is finite. In this direction, recently, Bahmanpour et al. in \cite{BNS2015} proved that for any $\fa$-torsion $R$-module $M$ that $\Ext^i_R(R/\fa, M)$ is finitely generated for all $i\leq \dim M$, all Bass numbers $\mu^j(\fp,M)$ and all Betti numbers $\beta_j(\fp,M)$ of $M$ are finite. In this paper, as a generalization of this result, we will prove that the assertion in this result holds when we replace "finitely generated" by "minimax". More precisely, we shall show that:

\begin{Coro}
Let $M$ be an $R$-module of dimension $n$ such that $\Supp_R(M)\subseteq V(\fa)$ and $\Ext^i_R(R/\fa,M)$ is minimax for all $i\leq n$. Then all Bass numbers $\mu_i (\fp,M)$ and all Betti numbers $\beta_i(\fp, M)$ of $M$ are finite.
\end{Coro}

Our main tools for proving this result is the following theorem which states some conditions for the cominimaxness of local cohomology modules with respect to ideals of dimension one.

\begin{theo}\label{In7}
Let $M$ be an $R$-module such that $\Ext^i_R(R/\fa,M)$ is minimax for all $i\leq \dim M$. Then the following assertions hold:
\begin{enumerate}
  \item The $R$-module $H^i_\fb(M)$ is $\fb$-cominimax for all $i\geq 0$ and for any ideal $\fb\subseteq \fa$ with $\dim R/\fb\leq1$.
  \item The $R$-module $\Ext^i_R(N,M)$ is minimax for all $i\geq 0$ and for any finitely generated $R$-module $N$ with $\Supp_R(N) \subseteq V (\fa)$ and $\dim N \leq 1$.
\end{enumerate}
\end{theo}

The proof of Theorem \ref{In7} is given in Proposition \ref{n6} and Theorem \ref{n7}.

Throughout the paper, we assume that $R$ is a commutative Noetherian ring, $\fa$ is an ideal of $R$ and $V(\fa)$ is the set of all prime ideals of $R$ containing $\fa$. For any unexplained notation and terminology we refer the reader to \cite{Matsu}.

\section{MAIN RESULTS}

Recall that a class of $R$-modules is a \emph{Serre subcategory} of the category of $R$-modules when it is closed under taking submodules, quotients and extensions. For example, the classes of Noetherian modules, Artinian modules or minimax modules are Serre subcategories. As in standard notation, we let $\mathcal{S}$ stand for a Serre subcategory of the category of $R$-modules.
The following lemma which is needed in the sequel, immediately follows from the definition of $\Ext$ and $\Tor$ functors.

\begin{lemma}\label{pro1}
Let $M$ be a finitely generated $R$-module and $N\in \mathcal{S}$. Then $\Ext^i_
R(M,N)\in \mathcal{S}$ and $\Tor_i^R (M,N)\in \mathcal{S}$ for all $i\geq0$.
\end{lemma}

\begin{lemma}\label{Gruson}
Suppose that $M$ is a finitely generated $R$-module and $N$ is an arbitrary $R$-module. Let for some $t\geq0$, $\Ext^i_R(M,N)\in \mathcal{S}$ for all $i\leq t$. Then $\Ext^i_R(L,N)\in \mathcal{S}$ for all $i\leq t$ and any finitely generated $R$-module $L$ with $\Supp_R(L)\subseteq\Supp_R(M)$.
\end{lemma}
\begin{proof}
See \cite[Lemma 2.2]{AR}.
\end{proof}

Let us mention some elementary properties of the minimax modules that we shall use.

\begin{remark}\label{pro}
The following statements hold:
\begin{enumerate}
  \item The class of minimax modules contains all finitely generated and all Artinian modules.\label{i}
  \item Let $0\rightarrow L\rightarrow M\rightarrow N\rightarrow0$ be an exact sequence of $R$-modules. Then $M$ is minimax if and only if $L$ and $N$ are both minimax (see \cite[Lemma~2.1]{BN08}). Thus any submodule and quotient of a minimax module is minimax.\label{ii}
  \item The set of associated primes of any minimax $R$-module is finite.\label{iii}
  \item Every zero-dimensional minimax $R$-module is Artinian.\label{iv}
  \item If $R$ is a vector space, then every minimax $R$-module is finite length.\label{v}
  \end{enumerate}
\end{remark}

\begin{lemma}\label{minimax}
Let $\fa$ be an ideal of $R$, $M$ be an $R$-module and $n$ be a non-negative integer such that $\Ext^n_R(R/\fa, M)$ (resp. $\Ext^{n+1}_R(R/\fa, M)$) is minimax. If $H^i_\fa(M)$ is $\fa$-cominimax for all $i<n$, then $\Hom_R(R/\fa, H^n_\fa(M))$ (resp. $\Ext^1_R(R/\fa, H^n_\fa(M))$) is minimax.
\end{lemma}
\begin{proof}
Since the class of minimax modules is a Serre subcategory, the assertion follows from \cite[Lemma~2.3]{AbaB2015}.
\end{proof}

The following lemma is well-known for $\fa$-cominimax modules.

\begin{lemma}\label{n1}
Let $M$ be an $\fa$-torsion $R$-module such that $\dim M\leq 1$. Then $M$ is $\fa$-cominimax if and only if the $R$-modules $\Hom_R(R/\fa,M)$ and $\Ext^1_R(R/\fa,M)$ are minimax.
\end{lemma}
\begin{proof}
See \cite[Proposition 2.4]{Irani}.
\end{proof}

The following proposition which states some conditions for the cominimaxness of local cohomology modules with respect to ideals of dimension at most one, plays a key role for proving the next theorem and the main result of this paper.

\begin{prop}\label{n6}
Let $M$ be an $R$-module of dimension $n$ such that $\Ext^j_R(R/\fa,M)$ is minimax for all $j\leq n$. Then the $R$-module $H^i_\fb(M)$ is $\fb$-cominimax for all $i\geq 0$ and for any ideal $\fa\subseteq \fb$ with $\dim R/\fb\leq1$.
\end{prop}
\begin{proof}
By Grothendieck's Vanishing Theorem we only need to prove the assertion for $0\leq i \leq n$.
Let $\fb$ be an arbitrary ideal of $R$ containing $\fa$ with $\dim R/\fb\leq 1$. Then by assumption and Lemma \ref{Gruson}, $\Ext^j_R(R/\fb,M)$ is a minimax $R$-module for all $j\leq n$.
We first prove the assertion for the case $n=0$. Then by assumption, the $R$-module $$\Hom_R(R/\fb, \Gamma_\fb(M))=\Hom_R(R/\fb, M)$$ is minimax and so is Artinian by Remark \ref{pro}(\ref{iv}). Hence, $\Gamma_\fb(M)$ is Artinian by virtue of Melkersson's result \cite[Theorem 1.3]{Mel2}. Now, the assertion follows from the Grothendieck's Vanishing Theorem and the fact that $\Supp_R(\Gamma_\fb(M))\subseteq V(\fb)$ and the class of minimax modules contains all Artinian modules. Thus, it remains to give the proof for the case $n>0$. For this purpose, there are two cases to consider: $\dim R/\fb=0$ or $\dim R/\fb=1$.

\textbf{Case 1:} If $\dim R/\fb=0$, then in the light of assumption and Remark \ref{pro}(\ref{iv}), $\Hom_R(R/\fb, M)$ is Artinian. Hence by the argument in the case of $n=0$, we may conclude that $\Gamma_\fb(M)$ is $\fb$-cominimax. Now suppose, inductively, that $0 < i \leq n$ and the $R$-modules $$H^0_\fb(M), H^1_\fb(M), \cdots, H^{i-1}_\fb(M)$$ are $\fb$-cominimax. Since $\Supp_R(H^i_\fb(M))\subseteq V(\fb)$ and the $R$-module $\Ext^j_R(R/\fb, M)$ is minimax for all $j \leq n$, we infer from Lemma \ref{minimax} that $\Hom_R(R/\fb, H^i_\fb(M))$ is a zero-dimensional minimax $R$-module and so is Artinian. Hence, $H^i_\fb(M)$ is Artinian by \cite[Theorem 1.3]{Mel2}, as desired.

\textbf{Case 2:} Let $\dim R/\fb=1$. The proof is by induction on $0 \leq i < n$. Since $\Hom_R(R/\fb, M/\Gamma_\fb(M))=0$, it follows from the assumption and the exact sequence
\begin{align*}
  &0 \rightarrow \Hom_R(R/\fb, \Gamma_\fb(M))\rightarrow \Hom_R(R/\fb,M)\rightarrow \Hom_R(R/\fb, M/\Gamma_\fb(M)) \\ & \rightarrow \Ext^1_R(R/\fb, \Gamma_\fb(M))\rightarrow \Ext^1_R(R/\fb,M)
\end{align*}
that  the $R$-modules $\Hom_R(R/\fb, \Gamma_\fb(M))$ and $\Ext^1_R(R/\fb, \Gamma_\fb(M))$ are minimax. Hence, as $\dim \Gamma_\fb(M)\leq 1$, the $R$-module  $\Gamma_\fb(M)$ is $\fb$-cominimax by Lemma \ref{n1}.
Now suppose that the assertion holds for $i-1$; we will prove it for $i$. By the inductive hypotheses, the $R$-modules $$H^0_\fb(M), H^1_\fb(M), \cdots, H^{i-1}_\fb(M)$$ are $\fb$-cominimax. Since the $R$-modules $\Ext^i_R(R/\fb,M)$ and $\Ext^{i+1}_R(R/\fb,M)$ are minimax, it follows from Lemma \ref{minimax} that the $R$-modules $$\Hom_R(R/\fb, H^i_\fb(M)\ \text{and}\  \Ext^1_R(R/\fb, H^i_\fb(M))$$  are minimax and so in view of Lemma \ref{n1} the $R$-module $H^i_\fb(M)$ is $\fb$-cominimax, for all $i = 0, 1, \cdots, n-1$.
Since $\Ext^n_R(R/\fb, M)$ is minimax, $\Hom_R(R/\fb, H^n_\fb(M))$ is also minimax by Lemma \ref{minimax}. If there exists $\fp \in \Supp_R(H^n_\fb(M))\subseteq V(\fb)$ with $\dim R/\fp=1$, then it is easy to see that $\dim M_\fp\leq n-1$ and so $(H^n_\fb(M))_\fp=0$ by Grothendieck's Vanishing Theorem, a contradiction.
Therefore, $$\Supp_R(H^n_\fb(M))\subseteq \Max(R).$$ This implies that the $R$-module $\Hom_R(R/\fb, H^n_\fb(M))$ is Artinian by Remark \ref{pro}(\ref{iv}). Hence, $H^n_\fb(M)$ is Artinian by \cite[Theorem 1.3]{Mel2} and so is $\fb$-cominimax, as required.
\end{proof}

\begin{theo}\label{n7}
Let $M$ be an $R$-module of dimension $n$ such that $\Ext^i_R(R/\fa,M)$ is minimax for all $i\leq n$. Then the $R$-module $\Ext^i_R(N,M)$ is minimax for all $i\geq 0$ and for any finitely generated $R$-module $N$ with $\Supp_R(N) \subseteq V(\fa)$ and $\dim N \leq 1$.
\end{theo}
\begin{proof}
Let $N$ be a finitely generated $R$-module such that $\Supp_R(N) \subseteq V (\fa)$
and $\dim N \leq 1$. Then, using \cite[Theorem 6.4]{Matsu}, there exist prime ideals $\fp_1, \cdots, \fp_t$ of $R$ and a chain $0=N_0\subseteq N_1\subseteq\cdots\subseteq N_t=N$ of submodules of $N$ such that $N_j/N_{j-1}\cong R/\fp_j$ for all $j = 1, \cdots,t$. Since $\fp_j \in \Supp_R(N)$, we deduce that $\dim R/\fp_j \leq 1$ and so in the light of Proposition \ref{n6}, the $R$-module $H^i_{\fp_j}(M)$ is $\fp_j$-cominimax for all $i \geq 0$ and for each $j = 1, \cdots, t$. Thus, by \cite[Corollary~3.10]{Mel}, the $R$-module $\Ext^i_R(R/\fp_j,M)$ is minimax for all $i \geq 0$ and for each $j = 1, \cdots, t$. Now, considering the exact sequences
\begin{align*}
 0\rightarrow N_1 \rightarrow &N_2 \rightarrow R/\fp_2 \rightarrow 0\\
   0\rightarrow N_2 \rightarrow &N_3 \rightarrow R/\fp_3 \rightarrow 0 \\
    &\vdots \\
     0\rightarrow N_{t-1} \rightarrow &N_t \rightarrow R/\fp_t \rightarrow 0
\end{align*}
we infer that $\Ext^i_R(N, M)$ is minimax, as desired.
\end{proof}

Now we are ready to state the main result of this paper.
Recall that for each $R$-module $M$, all integers $j \geq 0$ and all prime ideals $\fp$ of $R$, the $j$th \emph{Bass number} of $M$ with respect to $\fp$ is defined as $\mu^j(\fp,M) = \dim_{k(\fp)} \Ext^j_{R_\fp}
(k(\fp),M_\fp)$ and the $j$th \emph{Betti number} of $M$ with respect to $\fp$ is defined as $\beta_j(\fp,M) = \dim_{k(\fp)} \Tor_j^{R_\fp}(k(\fp),M_\fp)$, where $k(\fp) := R_\fp/{\fp R_\fp}$. Recently, Bahmanpour et al. in \cite[Corollary 2.7]{BNS2015} proved that for any $\fa$-torsion $R$-module $M$ that $\Ext^i_R(R/\fa, M)$ is finitely generated for all $i\leq \dim M$, all Bass numbers $\mu^j(\fp,M)$ and all Betti numbers $\beta_j(\fp,M)$ of $M$ are finite.
As an immediate consequence of Theorem \ref{n7} we obtain the next corollary which is a generalization of \cite[Theorem 1.9]{Mel1999} and \cite[Corollary~2.7]{BNS2015} and shows that the assertion in \cite[Corollary 2.7]{BNS2015} holds when we replace "finitely generated" by "minimax".

\begin{Coro}
Let $M$ be an $R$-module of dimension $n$ such that $\Supp_R(M)\subseteq V(\fa)$ and $\Ext^i_R(R/\fa,M)$ is minimax for all $i\leq n$. Then all Bass numbers $\mu_i (\fp,M)$ and all Betti numbers $\beta_i(\fp, M)$ of $M$ are finite.
\end{Coro}
\begin{proof}
Let $\fp\in\Spec(R)$ and $k(\fp) = R_\fp/\fp R_\fp$ be the residue field of $R_\fp$. If $\fp \notin V(\fa)$, then $M_\fp=0$ by assumption and there is nothing to prove. Otherwise, using Theorem \ref{n7} and letting $N:=R_\fp/\fp R_\fp$, we conclude that the $R_\fp$-module $\Ext^i_{R_\fp}(k(\fp),M_\fp)$ is minimax for all $i\geq 0$. Since $\Ext^i_{R_\fp}(k(\fp),M_\fp)$ is also a $k(\fp)$-vector space, it must be of finite length by Remark \ref{pro}(\ref{v}).  Now the proof is completed by \cite[Theorem 2.1]{Mel}.
\end{proof}

Consequently, we get the following equivalent conditions for the cominimaxness of local cohomology modules with respect to ideals of dimension at most one.

\begin{Coro}\label{cf}
Let  $M$ an $R$-module of dimension $n$ and $\fa$ be an ideal of $R$ such that $\dim R/\fa \leq 1$. Then the following conditions are equivalent:

\begin{enumerate}
  \item $\Ext^i_R(R/\fa,M)$ is minimax for all $i\leq \dim M$;
  \item $H^i_\fa(M)$ is $\fa$-cominimax for all $i$;
  \item $\Ext^i_R(R/\fa,M)$ is minimax for all $i$;
  \item $\Ext^i_R(X,M)$ is minimax for all $i\leq \dim M$ and for any finitely generated
$R$-module $X$ with $\Supp_R(X)\subseteq V (\fa)$;
  \item $\Ext^i_R(X,M)$ is minimax for all $i\leq \dim M$ and for any finitely generated
$R$-module $X$ with $\Supp_R(X)= V (\fa)$;
  \item $\Ext^i_R(X,M)$ is minimax for all $i$ and for any finitely generated
$R$-module $X$ with $\Supp_R(X)\subseteq V (\fa)$;
  \item $\Ext^i_R(X,M)$ is minimax for all $i$ and for any finitely generated
$R$-module $X$ with $\Supp_R(X)= V (\fa)$.
\end{enumerate}
\end{Coro}
\begin{proof}
The assertions follow from Proposition \ref{n6}, Theorem \ref{n7} and Lemma \ref{Gruson}.
\end{proof}

\begin{Coro}
If $(R, \fm)$ is a local ring and $\fa$ be an ideal of $R$ such that $\dim R/\fa = 1$, then the following conditions are equivalent:
\begin{enumerate}
\item $H^i_\fa(M)$ is $\fa$-cominimax for all $i$;
\item $\mu_i (\fp,M)$ is finite for all $\fp\in V(\fa)$ and for all $i\leq \dim M$;
\item $\mu_i (\fp,M)$ is finite for all $\fp\in V(\fa)$ and for all $i$.
\end{enumerate}
\end{Coro}
\begin{proof}
It yields from \cite[Theorem~2.3]{BNSBass} and Corollary \ref{cf}.
\end{proof}


\providecommand{\bysame}{\leavevmode\hbox to3em{\hrulefill}\thinspace}
\providecommand{\MR}{\relax\ifhmode\unskip\space\fi MR }
\providecommand{\MRhref}[2]{%
  \href{http://www.ams.org/mathscinet-getitem?mr=#1}{#2}
}
\providecommand{\href}[2]{#2}

\end{document}